\documentclass[10pt]{lad-35my}
\usepackage[cp1251]{inputenc}
\usepackage[russian]{babel}
\usepackage[T2A]{fontenc}
\usepackage{amsthm, amsmath, amssymb, amsbsy, amsfonts, latexsym, euscript, amsfonts,mathrsfs}
\usepackage{calrsfs}
\usepackage{upgreek}
\usepackage{graphicx}
\usepackage{color}
\usepackage{wasysym}
\usepackage{enumerate}
\usepackage{xcolor} 

\usepackage[pdftex, unicode,colorlinks,
linkcolor=blue,citecolor=red,bookmarksopen,pdfhighlight=/N]{hyperref}

\def\leq{\leqslant}

\def\geq{\geqslant}

\newtheorem{theorem}{\bf Теорема}
\newtheorem*{theoB}{\bf Теорема Безиковича о покрытии}
\newtheorem{proposition}{\bf Предложение}
\newtheorem{corollary}{\bf Следствие}

\theoremstyle{remark}
\newtheorem{remark}{\bf Замечание}
\newtheorem{definition}{\bf Определение}
\newtheorem{example}{\bf Пример}
\numberwithin{equation}{section}

\begin{document}

\thispagestyle{plain}

$$
{}
$$


\begin{center}
{\Large \bf ОЦЕНКИ СНИЗУ СУБГАРМОНИЧЕСКИХ ФУНКЦИЙ
\\[4pt]
ЧЕРЕЗ РАССТОЯНИЕ ГАРНАКА}
\end{center}
\vskip 10mm

\noindent
\parbox{80mm}{
\begin{center}{\Large Б. Н. Хабибуллин}
\\[4pt]
{\small Башкирский государственный университет}
\\[-2pt]
{\small Уфа, Россия}
\\[-2pt]
{\small khabib-bulat@mail.ru}
\end{center}
}
\hfill
\parbox{70mm}{
\begin{center}
{\Large Э. У. Таипова}
\\[4pt]
{\small
Институт математики с
ВЦ УФИЦ РАН}
\\[-2pt]
{\small  Уфа, Россия}
\\[-2pt]
{\small elina.taipova@bk.ru}
\end{center}
}

\vskip5mm

\markboth{{\thepage}\hfill{\small \bf
Б. Н. Хабибуллин, Э. У. Таипова}}
{{\small \bf
Оценки снизу субгармонических функций через расстояние Гарнака}
\hfill{\thepage}}


\renewcommand{\thefootnote}{}
\footnotetext{
$\copyright$\;\;
Б. Н. Хабибуллин, Э. У. Таипова, 2021}

\vskip 5mm

\parbox{146mm}{\noindent
{\bf Аннотация.}
Пусть $G$ ---  произвольная  непустая ограниченная область в конечномерном евклидовом пространстве. Основные результаты --- оценки  снизу общего характера в точках из $G$ для произвольной  субгармонической функции $u\not\equiv -\infty$ на замыкании области $G$ через максимум функции $u$ на границе области $G$. Эти результаты новые и для плоских областей $G$, а для интервалов $G$ на числовой прямой также  ранее не отмечались. Они показывают, что ключевую роль в этих оценках играет расстояние Гарнака. Дальнейшие применения  к  субгармоническим, выпуклым, голоморфным функциям, а также к мероморфным функциям и разностям субгармонических функций в областях конкретного вида предполагается изложить в продолжении этой  работы

Let $G$ be a nonempty bounded domain in a finite-dimensional Euclidean space. The main results are general estimates from below at points from $G$ for an arbitrary subharmonic function $u\not\equiv -\infty$ on the closure of the domain $G$ through the maximum of the function $u$ on the boundary of the domain $G$. These results are new for planar domains $G$, and for intervals of $G$ on the numerical line have also not been previously noted. They show that the Harnack distance plays a key role in these estimates. Further applications to subharmonic, convex, holomorphic functions, as well as to meromorphic functions and differences of subharmonic functions in domains of a particular type are supposed to be outlined in the continuation of this article.

}

\vskip10mm

\section{Введение}
\subsection{Необходимые определения,  сведения и соглашения}
Одноточечные множества $\{x\}$ часто записываем без фигурных скобок, т.е. просто как $x$. 
${\mathbb N}$:=\{1,2, \dots\} и ${\mathbb R}$ 
---  множества {\it натуральных\/} и {\it вещественных\/} чисел,  
а ${\mathbb R}^+:=\bigl\{x\in {\mathbb R}\bigm| x\geq 0\bigr\}$ --- множество  {\it положительных\/} чисел. 
Всюду далее ${\tt d}\in {\mathbb N}$ --- размерность {\it евклидова пространства\/} 
${\mathbb R}^{\tt d}$ с {\it евклидовой нормой\/}  $|x|:=\sqrt{x_1^2+\dots +x_{\tt d}^2}$ 
для $x:=(x_1,\dots ,x_{\tt d})\in {\mathbb R}^{\tt d}$ и функцией ${\sf dist}(\cdot, \cdot)$  {\it евклидова расстояния,\/}
 а также с  {\it одноточечной компактификацией Александрова\/} 
${\mathbb R}_{\text{\tiny $\infty$}}^{\tt d}:={\mathbb R}^{\tt d}\cup \infty$. Так, при ${\tt d=1}$  это  даёт ${\mathbb R}_{\text{\tiny $\infty$}}$,  отличающееся от {\it расширенной числовой прямой\/} $\overline {\mathbb R}:=-\infty\cup {\mathbb R}\cup +\infty$, полученной порядковым пополнением с добавлением двух концов $+\infty:=\sup {\mathbb R}$
и $-\infty:=\inf {\mathbb R}$, а $\overline {\mathbb R}^+:={\mathbb R}^+\cup+\infty$.  Для {\it пустого множества\/} $\emptyset$ по определению $\inf \varnothing :=+\infty$ и  $\sup \varnothing :=-\infty$.
Через $x^+:=\sup \{0,x\}$ обозначаем {\it положительную часть\/} от 
 $x\in \overline {\mathbb R}$, а $x^-:=(-x)^+$ --- его {\it отрицательная часть\/}.    Вообще всюду далее
{\it положительность\/} --- это $\geq 0$, а {\it отрицательность\/} --- это $\leq 0$. 

Для  $S\subset {\mathbb R}^{\tt d}$
через $\complement S:={\mathbb R}_{\text{\tiny $\infty$}}^{\tt d}\setminus S$, $\overline S$, $S^\circ$ 
 и  $\partial S$ обозначаем соответственно {\it дополнение, замыкание, внутренность\/} и {\it границу\/} множества $S$ в  
${\mathbb R}_{\text{\tiny $\infty$}}^{\tt d}$. Всюду   $D\neq \emptyset$ --- 
{\it область\/} в ${\mathbb R}^{\tt d}$, т.е. открытое связное подмножество  в ${\mathbb R}^{\tt d}$.  
При $r\in \overline{\mathbb R}^+$ и  $x\in {\mathbb R}^{\tt d}$ через $B_x(r):=\bigl\{y\in {\mathbb R}^{\tt d}\bigm| |y-x|<r\bigr\}$,   $\overline  B_x(r):=\bigl\{y\in {\mathbb R}_{\text{\tiny $\infty$}}^{\tt d}\bigm| |y-x|\leq r\bigr\}$ и  $\partial \overline B_x(r):=\overline B_x(r)\setminus  B_x(r)$ обозначаем соответственно {\it открытый шар,\/} {\it замкнутый шар\/} и {\it сферу радиуса $r$ с центром\/} $x$ с  {\it площадями поверхностей единичных сфер\/} $\partial \overline  B_0(1)\subset {\mathbb R}^{\tt d}$, выражаемыми 
через  гамма-функцию Эйлера $\Gamma$  как   
\begin{equation}\label{{kK}s}
s_{\tt d-1}=\frac{2\pi^{\tt d/2}}{\Gamma (\tt d/2)}, 
\quad s_{\tt 0}=2, \;
s_{\tt 1}=2\pi, \; s_{\tt 2}=4\pi,\;  s_{\tt 3}=\pi^2, \dots  . 
\end{equation}
Ключевая роль зависящей от размерности ${\tt d}$ пространства ${\mathbb R}^{\tt d}$ {\it возрастающей\/} функции 
\begin{equation}\label{kKd-2}
\Bbbk_{\tt d-2} \colon  t\underset{0<t\in {\mathbb R}^+}{\longmapsto} \begin{cases}
t &\text{\it  при ${\tt d=1}$},\\
\ln t  &\text{\it  при ${\tt d=2}$},\\
 -\dfrac{1}{t^{\tt d-2}} &\text{\it при ${\tt d>2}$,} 
\end{cases} 
\qquad \Bbbk (0):=\lim_{0<t\to 0} \Bbbk (t)\in \overline {\mathbb R},
\end{equation}
в теории потенциала и субгармонических функций обусловлена тем, что при каждом фиксированном $x\in {\mathbb R}^{\tt d}$ порождённая ею  функция 
\begin{equation}\label{kd0}
y\underset{y \in {\mathbb R}^{\tt d}}{\longmapsto}
\frac{1}{s_{\tt d-1}\widehat{\tt d}}\Bbbk_{\tt d-2}\bigl(|y-x|\bigr), \qquad
\widehat{\tt d}:=\max\{{\tt 1,d-2}\}=1+({\tt d-3})^+\in {\mathbb N},
\end{equation}
--- это {\it фундаментальное решение уравнения Лапласа\/} $\bigtriangleup u=\updelta_x$ на ${\mathbb R}^{\tt d}$, где  $\bigtriangleup$   --- {\it оператор Лапласа,\/} действующий в смысле теории обобщённых функций, а $\updelta_x$ ---  вероятностная  {\it мера Дирака\/}  с носителем  ${\sf supp\,} \updelta_x=\{x\}$ в одной точке $x$. Функцию $\Bbbk_{\tt d-2}$ часто будем записывать без нижнего индекса  ${\tt d-2}$ как просто $\Bbbk$, когда  значения размерности 
${\tt d}\in {\mathbb N}$ ясны из формулировок  или контекста. 

Через  ${\sf har}(S)$ и ${\sf sbh}(S)$  обозначаем класс сужений на $S\subset {\mathbb R}^{\tt d}$ всех соответственно {\it гармонических\/}
(локально аффинных при ${\tt d}=1$) и {\it субгармонических\/} (локально выпуклых при ${\tt d}=1$) функций на каких-либо  открытых  окрестностях множества $S\subset {\mathbb R}^{\tt d}$ \cite{Rans}, \cite{HK}, \cite{Brelot}.

\begin{definition}[{\cite[1.3]{Rans}, \cite{Kohn}}]\label{dH}
 Для  области $D\subset {\mathbb R}^{\tt d}$ и конуса 
${\sf har}^+(D)\subset {\sf har}(D)$ всех {\it положительных\/} гармонических функций на $D$
{\it расстояние  Гарнака\/} между точками $x\in D$ и $y\in D$ 
на  $D$  обозначаем и определяем как
\begin{equation}\label{hard}
\hspace{-1mm}{\sf dist}_{{\sf har}}^D (x,y):=\inf\Bigl\{ d\in {\mathbb R}^+\Bigm|
\frac{1}{d}  h(y)\leq h(x) \leq d h(y)\quad  
\text{для всех  $h\in {\sf har}^+(D)$}\Bigr\}\geq 1.
\end{equation}
\end{definition}
Основные свойства расстояния Гарнака \eqref{hard} отражены в   \cite[1.3]{Rans}, \cite{Kohn}. 
  Например,  $\ln {\sf dist}_{{\sf har}}^D$ --- {\it полуметрика на $D$, непрерывная  относительно евклидовой метрики в $D$
\cite[\S~1]{Kohn},} которая совпадает с  {\it гиперболической метрикой\/} в случае  {\it односвязной\/} области 
 $D$ в {\it комплексной плоскости\/} ${\mathbb C}$, отличной от ${\mathbb C}$, при отождествлении ${\mathbb C}$ с ${\mathbb R}^{\tt 2}$ \cite{Chi12}. 

Наряду  с {\it симметричностью\/} имеет место {\it мультипликативное  неравенство треугольника\/}
\begin{equation}\label{pmt}
{\sf dist}_{{\sf har}}^D(x,y)\leq {\sf dist}_{{\sf har}}^D(x,o)\cdot {\sf dist}_{{\sf har}}^D(o,y)\text{ \it для любых $x,o,y\in D$}.
\end{equation}  
Для расстояния Гарнака справедлив  принцип подчинения  \cite[теорема 1.3.6]{Rans}, \cite[\S~3, 
теорема 3.3]{Kohn}, который будет использован лишь в частном  случае \cite[следствие  1.3.7]{Rans}, \cite[\S~3, 
теорема 3.2]{Kohn}: 
\begin{equation}\label{pr}
{\sf dist}_{{\sf har}}^G\leq {\sf dist}_{{\sf har}}^{D} \quad \text{\it  на области $D$ при $D\subset G$}.
\end{equation} 

Для шара расстояние Гарнака до центра шара выписывается явно \cite[\S~3, приложение]{Kohn}:
\begin{equation}\label{DB}
{\sf dist}_{{\sf har}}^{B_{o}(r)}(o, x)=\frac{\bigl(r+|x-o|\bigr)r^{{\tt d-2}}}{\bigl(r-|x-o|\bigr)^{{\tt d}-1}}
\quad\text{\it при  $x\in B_{o}(r)$}.
\end{equation}
Оно  возрастает при увеличении расстояния $|x-x_0|$ от точки $x$ до центра $x_0$ шара $B_{x_0}(R)$. 

Класс ${\sf sbh}(S)$ содержит и функцию $\boldsymbol{-\infty}$, равную сужению на $S$ функции,  тождественно равной $-\infty$ в некоторой открытой окрестности множества $S$. 
Отношение равенства $=$ на  ${\sf sbh}(S)$ для пары функций означает, что существует открытое множество в ${\mathbb R}^{\tt d}$, содержащее $S$, на котором обе эти функции определены и  совпадают поточечно.   
Если  $O\subset {\mathbb R}^{\tt d}$ --- {\it открытое множество,\/}   и субгармоническая функция 
 $u\in {\sf sbh}(O)$ такова, что $u\not\equiv-\infty$ на  каждой связной компоненте множества $O$, то  однозначно определена её  {\it мера Рисса\/} 
\begin{equation}\label{df:cm}
\varDelta_u\overset{\eqref{kd0}}{:=} \frac{1}{s_{\tt d-1}\widehat{\tt d}} {\bigtriangleup}  u, 
\end{equation}
являющаяся {\it мерой Радона\/} на $O$, т.е. мерой Бореля, конечной на компактах из $O$.
 Если $u\equiv -\infty$ на какой-нибудь связной компоненте множества $O$, то по определению мера Рисса $\varDelta_u$ принимает значение $+\infty$ на каждом подмножестве из этой связной компоненты. Это определение переносится и на функции  $u\in {\sf sbh}(S)$, субгармонические на борелевском множестве $S$, как сужение на $S$ меры Рисса субгармонической функции $u$ в некоторой открытой окрестности множества $S$. 

Для меры Бореля  
$\mu$  на ${\mathbb R}^{\tt d}$ и точки $x\in {\mathbb R}^{\tt d}$ 
\begin{equation}\label{muyr} 
\mu_x^{\text{\tiny \rm rad}}(t)\underset{t\in {\mathbb R}^+}{:=}\mu\bigl(\overline B_x(t) \bigr)\in \overline {\mathbb R}^+ 
\end{equation}
--- {\it радиальная  считающая функция меры $\mu$ с центром $x\in {\mathbb R}^{\tt d}$,\/}  {\it возрастающая\/} и {\it непрерывная справа\/} на ${\mathbb R}^+$  функция из ${\mathbb R}^+$ в $\overline {\mathbb R}^+$. 
 Кроме того, с  числом  $\widehat{\tt d}\overset{\eqref{kd0}}{=}1+({\tt d}-3)^+$ 
\begin{equation}\label{muyrN} 
{\sf N}_x^{\mu} (r)\underset{r\in \overline {\mathbb R}^+}{:=}\widehat{\tt d}\int_0^r\frac{\mu_x^{\text{\tiny \rm rad}}(t)}{t^{{\tt d}-1}} 
{\,{\mathrm d}} t\in \overline {\mathbb R}^+ 
\end{equation}
---  радиальная {\it усреднённая,\/} или проинтегрированная, считающая функция меры $\mu$ с центром $x\in {\mathbb R}^{\tt d}$,  {\it возрастающая\/} и {\it непрерывная\/} на ${\mathbb R}^+$ как функция из ${\mathbb R}^+$ в $\overline {\mathbb R}^+$.

\subsection{Основные результаты об оценках снизу  субгармонических функций}
В первом результате  оценивается снизу значение функции в одной точке области.  

Для $S\subset {\mathbb R}^{\tt d}$ через $\diameter\!S:=\sup\limits_{x,y\in S}|x-y|$ обозначаем
{\it евклидов диаметр} множества $S$. 

\begin{theorem}\label{th1_1}
Пусть $D$ --- ограниченная  область в $ {\mathbb R}^{\tt d}$ с фиксированной точкой  
$o\in D$,  а  $u$ --- субгармоническая функция на замыкании  $\overline D$  со значением  $u(o)=0$ и мерой Рисса $\varDelta_u$. Тогда для сужения    $\varDelta\overset{\eqref{df:cm}}{:=}\varDelta_u\!\bigm|_{\overline {D}}$
 её меры Рисса $\varDelta_u$ на замыкание $\overline {D}$ области $D$ имеем 
\begin{equation}\label{infuI20}
u(x)\geq -\bigl({\sf dist}_{{\sf har}}^{D}(o,x) -1\bigr)\sup_{\partial D} u
-{\sf N}_x^{\varDelta}(\diameter\!D)
\quad \text{для   каждой точки  $x\in D$}.
\end{equation}
Если для этой функции $u\in {\sf sbh}(\overline D)$ указана  область $G\subset {\mathbb R}^{\tt d}$, для которой
 \begin{equation}\label{oBDG+}
o\in  D\subset \overline D\subset G\subset \overline G\subset {\mathbb R}^{\tt d},\quad 
u\in {\sf sbh}(\overline G),\quad u(o)=0,
\end{equation}
 то при любых значениях $r_x\in (0, \diameter\!D]$ выполнено неравенство 
\begin{subequations}\label{rd}
\begin{gather}
\hspace{-1mm}u(x)\geq -\bigl({\sf dist}_{{\sf har}}^{D}(o,x) -1\bigr)\sup_{\partial D} u- 
\frac{\Bbbk(\diameter\!D)-\Bbbk(r_x)}{\Bbbk(R+\text{\dj})-\Bbbk(R)}\sup\limits_{y\in \partial D}{\sf dist}_{{\sf har}}^{G\!\setminus\!o}\bigl(y,\partial B_o(R)\bigr)\sup_{\partial G} u
 -{\sf N}_x^{\varDelta}(r_x),
\tag{\ref{rd}u}\label{infuI2}
\\
\intertext{где по приципу максимума $\sup\limits_{\partial D} u$ в правой части можно заменить 
на $\sup\limits_{\partial G} u$, а}
R:={\sf dist}(o,\partial D)
\tag{\ref{rd}R}\label{rdr}
\\
\intertext{--- радиус шара с центром $o$, вписанного в область $D$,}
\text{\dj} :={\sf dist}(D, \complement G).
\tag{\ref{rd}d}\label{rdd}
\end{gather}
\end{subequations}
---  евклидово расстояние от границы  $\partial D$ области $D$ до границы $\partial G$ области $G$.

При $u(x)\neq -\infty$ правые  части неравенств \eqref{infuI20} и  \eqref{infuI2} 
конечны. 
\end{theorem}
Неравенство \eqref{infuI20}
доказано в конце раздела \ref{Sec3}, а \eqref{rd} --- в центральной части  раздела \ref{Sec5}. 

\begin{remark} Можно  снять условие $u(o)=0$ при $u(o)\neq -\infty$, 
если заменить  $u(x)$ на разность  $u(x)-u(o)$ в левых  частях 
\eqref{infuI20} и \eqref{infuI2}, а также $\sup\limits_D u$ на $\sup\limits_D u-u(o)$ в правой части неравенства  \eqref{infuI20} и соответственно  $\sup\limits_G u$ на $\sup\limits_G u-u(o)$ в правой части  \eqref{infuI2}. С теми же заменами  неравенства  \eqref{infuI20} и \eqref{infuI2} верны и при $u(o)=-\infty$ в случае $u(x)\neq -\infty$, поскольку обе части неравенств  \eqref{infuI20} и \eqref{infuI2} корректно определены, а  в левой части будет $+\infty$. 
\end{remark}

Если рассматривать оценки снизу субгармонической функции $u\in  {\sf sbh}(\overline D)$ на некотором множестве, то, как правило,  целесообразно дать оценку вне некоторого исключительного малого множества без упоминания о мере Рисса функции или её сужениях. Мы приведём здесь  вариант таких оценок снизу, основанный  на так называемом методе нормальных точек, восходящем к оценке А.~Картана модуля многочлена снизу, для чего потребуется 

\begin{definition}[{\cite{Carleson}, \cite{Federer}, \cite{Rodgers}, \cite{HedbergAdams}, \cite{Eid07}, \cite{VolEid13}}]\label{defH}
Для  величины  $r\in \overline {\mathbb R}^+\setminus 0$  и произвольной  функции   $h\colon [0,r)\to {\mathbb R}^+$
функцию множеств 
\begin{equation}\label{mr}
{\mathfrak m}_h^{\text{\tiny $r$}}\colon S\underset{S\subset {\mathbb R}^{\tt d}}{\longmapsto}  \inf \Biggl\{\sum_{j\in N} h(r_j)\biggm| N\subset {\mathbb N},\,  S\subset \bigcup_{j\in N} 
\overline B_{x_j}(r_j), \, x_j\in {\mathbb R}^{\tt d}, \, r_j\underset{j\in N}{\leq} r\Biggr\} \in \overline {\mathbb R}^+
\end{equation}
называем {\it $h$-обхватом  Хаусдорфа радиуса\/} 
$r$. Для каждого $S\subset {\mathbb R}^{\tt d}$ значения    ${\mathfrak m}_h^{\text{\tiny $r$}}(S)$ убывают по $r$ и существует предел   
\begin{equation}\label{hH}
{\mathfrak m}_h^{\text{\tiny $0$}}(S):=\lim_{0<r\to 0} {\mathfrak m}_h^{\text{\tiny $r$}}(S)
\geq {\mathfrak m}_h^{\text{\tiny $r$}}(S)\geq {\mathfrak m}_h^{\text{\tiny $\infty$}}(S)
 \quad \text{\it для любого  $S\subset {\mathbb R}^{\tt d}$}. 
\end{equation}
При $h(0)=0$ все обхваты ${\mathfrak m}_h^{\text{\tiny $r$}}$ --- внешние меры,  а ${\mathfrak m}_h^{\text{\tiny $0$}}$ определяет  {\it $h$-меру Хаусдорфа\/} ${\mathfrak m}_h^{\text{\tiny $0$}}$, являющуюся регулярной мерой Бореля. Чаще всего используются   степенные   функции $h_p$ степени $p\in {\mathbb R}^+$   с нормирующим множителем вида
\begin{equation}\label{hd}
h_p\colon t\underset{t\in {\mathbb R}^+}{\longmapsto} 
c_pt^p, \quad\text{где } 
c_p:=\dfrac{\pi^{p/2}}{\Gamma(p/2+1)}, 
\end{equation}
а  $h_p$-обхват радиуса  $r$ и  $h_p$-меру Хаусдорфа 
называем соответственно  {\it $p$-мер\-н\-ы\-ми  обхватом радиуса  $r$\/} и\/ {\it мерой Хаусдорфа\/}, которые обозначаем соответственно как
\begin{equation}\label{p-m}
p\text{\tiny-}{\mathfrak m}^{\text{\tiny $r$}}:={\mathfrak m}_{h_p}^{\text{\tiny $r$}}, \quad  
p\text{\tiny-}{\mathfrak m}^{\text{\tiny $0$}}:={\mathfrak m}_{h_p}^{\text{\tiny $0$}}. 
\end{equation}
\end{definition}

Здесь и далее классические и широко известные свойства обхватов и мер Хаусдорфа из основных источников, указанных в  начале определения  \ref{defH}, часто  используются без явно прописанных  конкретных  ссылок. 

\begin{example}\label{ex:1} 
Пространственная $\tt d$-мерная {\it мера Лебега\/} 
 на ${\mathbb R}^{\tt d}$ в любой размерности  ${\tt d}\in {\mathbb N}$ совпадает с  
${\tt d}$-мерной  мерой Хаусдорфа ${\tt d}\text{\tiny-}{\mathfrak m}^{\text{\tiny $0$}}(S)$, а 
 $p$-мерная  мера   $p\text{\tiny-}{\mathfrak m}^{\text{\tiny $0$}}$ в ${\mathbb R}^{\tt d}$ 
при $p>{\tt d}$ нулевая.
\end{example}

\begin{corollary}\label{cor1}
Пусть выполнены условия теоремы {\rm \ref{th1_1}}, включая  \eqref{oBDG+}, а также
\begin{equation}\label{key}
S\subset \overline S\subset D, \quad 0<r\leq \diameter\!D.
\end{equation}
Тогда для любой функции $h\colon [0,r]\to {\mathbb R}^+$ с $h(0)=0$ и конечным интегралом\/ {\rm (ср. с \eqref{muyrN})}
\begin{equation}\label{N0h}
N_0^h(r):=\widehat{\tt d}\int_0^r\frac{h(s)}{s^{{\tt d}-1}}{\,{\mathrm d}} s<+\infty
\end{equation}
найдётся множество  $E\subset S$, для которого 
\begin{subequations}\label{eEr}
\begin{align}
 \inf_{x\in S\setminus E} u(x)&\geq -\biggl(
\sup_{x\in S}{\sf dist}_{{\sf har}}^{D}(o,x) -1
\notag
\\
&+\frac{\Bbbk(\diameter\!D)-\Bbbk(r)}{\Bbbk(R+\text{\dj})-\Bbbk(R)}\sup\limits_{x\in \partial D}{\sf dist}_{{\sf har}}^{G\!\setminus\!o}\bigl(x,\partial B_o(R)\bigr)+N_0^h(r)\biggr)
\sup_{\partial G} u,
\tag{\ref{eEr}u}\label{{eEr}u}
\\
\intertext{и в то же время имеет место оценка сверху на $h$-обхват радиуса $r$ множества $E$}
\mathfrak m_h^r(E)&\leq 
\frac{5^{\tt d}}{\Bbbk(R+\text{\dj})-\Bbbk(R)}\sup\limits_{x\in \partial D}{\sf dist}_{{\sf har}}^{G\!\setminus\!o}\bigl(x,\partial B_o(R)\bigr).
\tag{\ref{eEr}E}\label{{eEr}E}
\end{align}
\end{subequations}
\end{corollary}
Следствие  \ref{cor1}  доказано в конце раздела \ref{Sec5}. 
\begin{remark}
Чем больше $h$, тем слабее неравенство  \eqref{{eEr}u}, но тем и меньше исключительное множество $E$, и наоборот. В частности, при домножении $h$ на строго положительные эта зависимость обратно пропорциональная. 
Так, в случае  степенной функции \eqref{hd} с $p\in ({\tt d}-2,{\tt d}]$, дополненной числовым множителем $B\in {\mathbb R}^+\setminus 0$, для $N_0^h(r)$ в правой части \eqref{{eEr}u} получаем
 $$
N_0^h(r)\overset{\eqref{N0h}}{=}Bc_p\widehat{\tt d}\int_0^rt^{p-{\tt d}+1}{\,{\mathrm d}} t
=B\frac{c_p\widehat{\tt d}}{p-({\tt d}-2)}r^{p-({\tt d}-2)},
$$
в то время как для  $p$-мерного  обхвата \eqref{p-m} радиуса  $r$ исключительного множества $E$ имеем 
$$
p\text{\tiny-}{\mathfrak m}^{\text{\tiny $r$}}(E)\overset{\eqref{{eEr}E}}{\leq}
\frac{1}{B}\cdot\frac{5^{\tt d}}{\Bbbk(R+\text{\dj})-\Bbbk(R)}\sup\limits_{x\in \partial D}{\sf dist}_{{\sf har}}^{G\!\setminus\!o}\bigl(x,\partial B_o(R)\bigr).
$$

\end{remark}

Мы не обсуждаем необъятную информацию о предшествующих оценках снизу субгармонических функций и их широких применениях, а также  не рассматриваем оценки снизу в конкретных ограниченных областях (круг, шар, выпуклая область и т.п.) и переходы от них к неограниченным областям (плоскость, полуплоскость и (полу)полоса, пространство, полупространство и пр.) путём исчерпания ограниченными.

\section{Оценки снизу гармонических функций}

\begin{proposition}\label{th2_1}
Пусть  $o$ --- некоторая точка в области $D\subset {\mathbb R}^{\tt d}$ и $h\in {\sf har}(D)$. Тогда 
\begin{equation}\label{hm}
 \inf_S h-h(o)\geq
-\biggl(\sup_{x\in S} {\sf dist}_{{\sf har}}^D(o,x)-1\biggr)\Bigl(\sup_D h-h(o)\Bigr)
\quad\text{для любого $S\subset D$,}
\end{equation}
где обе перемножаемые скобки в правой части  положительны. 
\end{proposition}
\begin{proof}  Положительность первой скобки в правой части неравенства \eqref{hm} следует из 
\eqref{hard}, а положительность  второй скобки очевидна ввиду $o\in D$. 
Таким образом, если функция $h$ не ограничена в $D$, то справа получаем $-\infty$ и неравенство  \eqref{hm} верно. 
В случае  ограниченной сверху на $D$ функции $h$ положим  
\begin{equation}\label{M}
M:=\sup_D h \geq h(o).
\end{equation} 
Тогда $H:=M-h\in {\sf har}^+(D)$  и по определению \ref{dH} расстояния Гарнака 
для любой точки $x\in D$
\begin{equation*}
M-h(x)=H(x)\overset{\eqref{hard}}{\leq} {\sf dist}_{{\sf har}}^D(o,x) H(o)={\sf dist}_{{\sf har}}^D(o,x)\bigl(M-h(o)\bigr)
\end{equation*}
откуда $h(x)\geq M-{\sf dist}_{{\sf har}}^D(o,x)\bigl(M-h(o)\bigr)$ для любой точки $x\in S\subset D$.
Применение точной нижней  грани по $x\in S$ к обеим частям последнего неравенства  даёт 
\begin{subequations}\label{Mh}
\begin{gather}
\inf_S h\geq \inf_{x\in S}\Bigl(M-{\sf dist}_{{\sf har}}^D(o,x)\bigl(M-h(o)\bigr)\Bigr)
=M+\inf_{x\in S} \Bigl(-{\sf dist}_{{\sf har}}^D(o,x)\bigl(M-h(o)\bigr)\Bigr)
\tag{\ref{Mh}i}\label{Mhi}
\\
=M+\bigl(M-h(o)\bigr) \inf_{x\in S} \bigl(-{\sf dist}_{{\sf har}}^D(o,x)\bigr)=
M-\bigl(M-h(o)\bigr) \sup_{x\in S} {\sf dist}_{{\sf har}}^D(o,x)\bigr),
\tag{\ref{Mh}ii}\label{Mhii}
\end{gather}
\end{subequations}
где переход от \eqref{Mhi} к \eqref{Mhii} опирается на {\it положительность} $M-h(o)\overset{\eqref{M}}{\geq} 0$.
Вычитая $h(o)$ из крайних частей цепочки неравенств \eqref{Mh}, получаем  
\begin{equation*}
\inf_S h-h(o)\geq M-h(o)-\bigl(M-h(o)\bigr) \sup_{x\in S} {\sf dist}_{{\sf har}}^D(o,x)\bigr)
=\bigl(M-h(o)\bigr)\Bigl(1-\sup_{x\in S} {\sf dist}_{{\sf har}}^D(o,x)\Bigr),
\end{equation*}
откуда, возвращаясь  к обозначению \eqref{M}, приходим к  \eqref{hm}.
\end{proof}

\section{Оценки функции Грина сверху}\label{Sec3}

\subsection{Ёмкость и $\Bbbk$-потенциалы}
Для меры Радона $\mu$ на ${\mathbb R}^{\tt d}$
$$
{\sf pt}_{\mu}(y)\underset{y\in {\mathbb R}^{\tt d}}{:=}\int \Bbbk\bigl(|x-y|\bigr){\,{\mathrm d}} \mu(x)\in  \overline {\mathbb R}
$$
---  $\Bbbk$-потенциал меры $\mu$, который корректно определён и является субгармонической функцией, не равной тождественно $-\infty$ на ${\mathbb R}^{\tt d}$,  тогда и только  тогда, когда  ${\sf N}_x^{\mu} (+\infty)<+\infty$ для какой-нибудь точки $x\in {\mathbb R}^{\tt d}$. 
Для борелевского  $E\subset {\mathbb R}^{\tt d}$ определяется 
его {\it ёмкость\/} \cite{Rans}, \cite{HK} 
\begin{equation}\label{CapE}
\text{\sf Cap} (E) :=\inf_{ O= O^\circ\supset E}  
\sup_{{K=\overline K\subset O}} 
 \Biggl\{\Bbbk^{-1}\left(\int {\sf pt}_{\mu}(y) {\,{\mathrm d}} \mu(y) \right)\Biggm|
\begin{array}{c}
{\sf supp\,} \mu\subset K, \; \mu(K)=1,\\
 \text{$\mu$ --- мера Бореля} 
\end{array}\Biggr\}\in \overline {\mathbb R}^+.
\end{equation}

В настоящей  статье  нам достаточно рассматривать субгармонические функции  лишь на {\it связных 
 борелевских подмножествах $S\subset {\mathbb R}^{\tt d}$ с непустой внутренностью\/} $S^\circ\neq \emptyset$. 
Ёмкость таких множеств $S$ всегда ненулевая, т.е. $\text{\sf Cap} (S)>0$.  Для функции $u\in {\sf sbh} ( S)$ через 
\begin{equation}\label{keyiu}
{(-\infty)}_{u}(S):=\bigl\{x\in S \bigm| u(x)=-\infty\bigr\}
\end{equation} обозначаем  {\it $(-\infty)$-мно\-ж\-е\-с\-т\-во\/} функции $u$  в $S$ \cite[3.5]{Rans}, где зачастую пишем просто ${(-\infty)}_{u}$, не указывая $S$.  
Если $u\in {\sf sbh}(S)\setminus \boldsymbol{-\infty}$, то  её $(-\infty)$-мно\-ж\-е\-с\-т\-во в $S$ --- это $G_{\delta}$-множество, т.е. пересечение не более чем счётного множества  открытых множеств,  нулевой  емкости $\text{\sf Cap}\bigl((-\infty)_u\bigr)=0$ \cite[теоремы 5.10, 5.32]{HK}. 
Подмножество $E\subset S$ {\it полярное,\/} если существует   функция $u\in {\sf sbh}(S)\setminus \boldsymbol{-\infty}$, для которой  $E\subset {(-\infty)}_{u}$, а (внешняя) ёмкость полярного множества равна нулю,  и обратное тоже  верно \cite[5.9, теорема 5.32]{HK}.  

Для функции ограниченной вариации $m$ на интервале в $\overline {\mathbb R}$, содержащем $(a,b]$, 
{\it интеграл\/} (Римана\,-- или  Лебега\,--\,){\it Стилтьеса\/}
по интервалу с концами $a,b$  понимаем как интеграл по  открытому слева и замкнутому справа интервалу $(a,b]\subset \overline {\mathbb R}$, если не оговорено иное:
\begin{equation}\label{keyint}
\int_a^b \dots {\,{\mathrm d}} m:=\int_{(a,b]} \dots {\,{\mathrm d}} m.
\end{equation}

\begin{proposition}\label{prkN}
Если $\mu$ --- мера Бореля с носителем ${\sf supp\,} \mu \subset \overline B_x(r)\subset {\mathbb R}^{\tt d}$, то  
\begin{equation}\label{reprk}
\int_{\overline B_x(r)}\Bigl( \Bbbk (r)-\Bbbk \bigl(|y-x|\bigr)\Bigr){\,{\mathrm d}} \mu(y) \overset{\eqref{keyint}}{\underset{x\in {\mathbb R}^{\tt d}}{=}}
\int_0^r\bigl( \Bbbk (r)-\Bbbk (t)\bigr){\,{\mathrm d}} \mu_x^{\text{\tiny \rm rad}}(t) \underset{x\in {\mathbb R}^{\tt d}}{=}
{\sf N}_x^{\mu} (r)\in \overline {\mathbb R}^+ ,\quad r\in {\mathbb R}^+\setminus 0,
\end{equation}
где  значения интегралов  конечны  тогда и только тогда  $x\notin (-\infty)_{{\sf pt}_{\mu}}=\Bigl\{  x\in {\mathbb R}^{\tt d}\Bigm| {\sf N}^{\mu}_x(r)=+\infty\Bigr\}$.
\end{proposition}
Доказательство опускаем, поскольку  \eqref{reprk} легко получается интегрированием по частям,  как в \cite[\S~2]{KhaShm20}, с обоснованиями в рамках классической теории потенциала \cite{Rans}, \cite{HK}, \cite{Brelot}, \cite{Landkof}. 

\subsection{Функция Грина, её потенциал и оценка их сверху} 
Пусть  $D\neq \varnothing $ --- {\it ограниченная область в ${\mathbb R}^{\tt d}$.}   Для  любой точки $x\in D$ однозначно  определена  продолженная {\it функция Грина  ${\sf g}^D_x\colon {\mathbb R}_{\text{\tiny $\infty$}}^{\tt d}\to \overline {\mathbb R}^+$ для области $D$ с полюсом в точке $x\in D$\/} \cite[5.7.4]{HK}, \cite[гл.~IX, \S~2]{Brelot}, 
равная {\it полунепрерывной сверху регуляризации точной верхней грани всех  субгармонических функций $v$ на ${\mathbb R}_{\text{\tiny $\infty$}}^{\tt d}\setminus x$, тождественно равных нулю на дополнении $\complement \overline D$, для которых  
\begin{equation}\label{gx0}
\limsup\limits_{x\neq y\to x}\frac{v(y)}{-\Bbbk(|y-x|)}\leq 1, 
\end{equation}
а  в точке $x$ полагаем    ${\sf g}^D_x(x):=+\infty$.\/} 
Вместе с последним равенством  
функция Грина  ${\sf g}^D_x$   полностью, однозначно и даже с избытком характеризуется  свойствами  
\begin{enumerate}[{\rm (i)}]
{\it
\item\label{ig} ${\sf g}^D_x\in {\sf sbh}^+({\mathbb R}_{\text{\tiny $\infty$}}^{\tt d}\setminus x)$ --- субгармоническая положительная функция вне полюса $x$,
\item\label{ig0} ${\sf g}^D_x(y)\equiv 0$ на $\complement \overline D$ ---  обращается в нуль всюду вне замыкания $\overline  D$,
\item\label{iiig} ${\sf g}_x^D(y)=-\Bbbk\bigl(|y- x|\bigr)+O(1)$ при $x\neq y\to x$, {\rm что сильнее \eqref{gx0}},
\item\label{iig} $-{\sf g}^D_x\bigm|_{D}\in {\sf sbh}(D)$, {\rm т.е.}  ${\sf g}^D_x$  --- супергармоническая функция на $D$, 
\item\label{ivg} ${\sf g}^D_x(y)={\sf g}^D_y(x)$ при  $y\in D$ --- симметричность функции Грина на $D$.
\item\label{vig} ${\sf g}^D_x\in {\sf har}\Bigl((D\setminus x)\cup \complement \overline D \Bigr)$  --- гармоничность  
функции Грина вне границы $\partial D$ и полюса $x$,}
\end{enumerate} 
где последнее свойство \eqref{vig} --- результат пересечения трёх свойств \eqref{ig}, \eqref{ig0}, \eqref{iig}. Для функции Грина справедлив принцип подчинения:  
{\it если  области $D$ и $G$ ограничены, то
\begin{equation}\label{prpodG}
{\sf g}^D_x(y)\leq {\sf g}^G_y(x) 
\text{при $D\subset G$ для любых  $x\in D$ и $y\in {\mathbb R}_{\text{\tiny $\infty$}}^{\tt d}$}.
\end{equation}
}
Для  $x\in {\mathbb R}^{\tt  d}$ и  $0<r\in {\mathbb R}^+$ 
\begin{equation}\label{Byr}
{\mathsf g}_x^{B_x(r)}(y)=\Bigl(\Bbbk(r)-\Bbbk\bigl(|y-x|\bigr)\Bigr)^+
\quad\text{\it при всех $y\in {\mathbb R}_{\text{\tiny $\infty$}}^{\tt d}$}.
\end{equation}

Вероятностная {\it гармоническая  мера\/ $\updelta_x^{\complement D} $ для $D$ в точке\/}   $x\in D$ 
--- это мера Рисса субгармонической на ${\mathbb R}_{\text{\tiny $\infty$}}^{\tt d}\setminus x$ функции Грина 
${\sf g}^D_x\overset{\eqref{ig}}{\in} {\sf sbh}^+({\mathbb R}_{\text{\tiny $\infty$}}^{\tt d}\setminus x)$ для $D$ с полюсом $x\in D$: 
\begin{equation}\label{upg}
\updelta_x^{\complement D}:=\varDelta_{{\sf g}^D_x}\overset{\eqref{df:cm}}{=} \frac{1}{s_{\tt d-1}\widehat{\tt d}}\bigtriangleup\!{\sf g}^D_x, \quad {\sf supp\,} \updelta_x^{\complement D}\overset{\eqref{vig}}{\subset} \partial D, 
\quad \updelta_x^{\complement D}({\mathbb R}^{\tt d})=1,  \quad x\in D,
\end{equation}
где обозначение $\updelta_x^{\complement D}$ в отличие от более распространённого  $\upomega_x^D$ обусловлено тем, что гармоническую меру для $D$ в  $x\in D$ можно получить и как {\it выметание меры Дирака $\updelta_x$ 
на дополнение $\complement D$ относительно конуса $ C(\overline D)\bigcap {\sf har} (D)$} \cite[гл. IV, \S\S~1--3]{Landkof}
в обозначении  $C(S)$ для класса непрерывных функций на $S$. 
Обращение равенства \eqref{upg} --- это представление для функции Грина (см. \cite[теорема 4.4.7]{Rans}, \cite[(5.7.13)]{HK}, или, более общ\'о, \cite[\S~1]{Kha03})
\begin{equation}\label{repg}
{\sf g}^D_x(y)=\int_{\partial D} \Bbbk\bigl(|y-z|\bigr){\,{\mathrm d}} \updelta_x^{\complement D}(z)- 
\Bbbk\bigl(|y-x|\bigr)\text{ при любых $x\in D$ и $y\in  {\mathbb R}^{\tt d}$}.
\end{equation} 
\begin{proposition}\label{pr_g}
Пусть   $D\neq \varnothing$ --- ограниченная область в ${\mathbb R}^{\tt d}$.  Тогда 
\begin{equation}\label{gek0}
{\sf g}_x^D(y)
\leq \Bbbk(\diameter\! D)- \Bbbk\bigl(|y-x|\bigr)
\quad\text{при  любых  $x\in D$ и $y\in \overline D$.}
\end{equation}
\end{proposition}
\begin{proof} 
Из представления \eqref{repg} получаем
$$
{\sf g}_x^D(y) \overset{\eqref{repg}}{\leq} 
\updelta_x^{\complement D}({\mathbb R}^{\tt d})\sup_{z\in \partial D} \Bbbk \bigl(|y-z|\bigr) - \Bbbk\bigl(|y-x|\bigr)
\overset{\eqref{upg}}{\leq} \sup_{\stackrel{z\in \partial D}{y\in \overline D}} \Bbbk \bigl(|y-z|\bigr) - \Bbbk\bigl(|y-x|\bigr).
$$
Здесь  для  любых $z\in \partial D$ и  $ y\in \overline D$ имеет место неравенство   $|y-z|\leq \diameter\!D$  и в силу возрастания функции $\Bbbk$, определённой в  \eqref{kKd-2},  получаем   \eqref{gek0}.
\end{proof}
\begin{remark}
Другое доказательство неравенства \eqref{gek0} можно получить  из принципа подчинения для функций Грина 
\eqref{prpodG} и явного вида функции Грина \eqref{Byr} для специальных шаров. 
\end{remark}

\begin{proposition}\label{prpG}
Если $\mu$ --- мера Радона  на $ {\mathbb R}^{\tt d}$ с сужением $\mu\!\bigm|_{\overline D}$ на замыкание $\overline D$ ограниченной области $D$, то  
имеют место неравенства для потенциала Грина меры $\mu$ 
по $D\subset {\mathbb R}^{\tt d}$ 
\begin{equation}\label{repr}
\int {\sf g}_x^D(y){\,{\mathrm d}} \mu(y)\leq
{\sf N}_x^{\mu|_{\overline D}} (\diameter\!D )\in \overline {\mathbb R}^+ \quad\text{при всех $x\in D$}.
\end{equation}
\end{proposition}
\begin{proof}
По свойству  \eqref{ig0} потенциал Грина меры $\mu$ равен 
$$
\int_{\overline D}{\sf g}_x^D(y){\,{\mathrm d}} \mu(y)
\overset{\eqref{ig0}}{=}\int_{\overline D}{\sf g}_x^D(y){\,{\mathrm d}} \mu\!\bigm|_{\overline D}(y)
\leq \int_{\overline B_x(\diameter\!D)}{\sf g}_x^D(y){\,{\mathrm d}} \mu\!\bigm|_{\overline D}(y)
\quad\text{при любых $x\in D$},
$$
поскольку $\overline D\subset \overline B_x(\diameter\!D)$, откуда по предложению \ref{pr_g}
\begin{equation*}
\int_{\overline D}{\sf g}_x^D(y){\,{\mathrm d}} \mu(y)\overset{\eqref{gek0}}{\leq} 
\int_{\overline B_x(\diameter\!D)}
 \Bigl(\Bbbk(\diameter\! D)- \Bbbk\bigl(|y-x|\bigr)\Bigr)
{\,{\mathrm d}} \mu\!\bigm|_{\overline D}(y)
\overset{\eqref{repr}}{=}
{\sf N}_x^{\mu|_{\overline D}} (\diameter\!D)\in \overline {\mathbb R}^+
\end{equation*}
для каждой точки   $x\in D$.
\end{proof}

\begin{proof}
[Доказательство неравенства \eqref{infuI20} теоремы \ref{th1_1}]
Функция 
\begin{equation}\label{hM}
h:= {\mathsf H}_u^D\colon x\underset{x\in D}{\longmapsto} \int_{\partial D}u {\,{\mathrm d}} \updelta_x^{\complement D}
\end{equation}
--- {\it наилучшая гармоническая мажоранты  функции $u$ внутри $D$\/} 
\cite[гл.~IX, \S~6]{Brelot}, \cite[5.7]{HK},  для которой 
по предложению \ref{th2_1} при одноточечном $S:=x\in D$ имеет место оценка снизу 
\begin{equation}\label{hmh}
  h(x)-h(o)\overset{\eqref{hm}}{\geq}
-\bigl( {\sf dist}_{{\sf har}}^D(o,x)-1\bigr)\Bigl(\sup_D h-h(o)\Bigr)
\quad\text{для любой точки $x\in  D$,}
\end{equation}
где обе скобки справа положительны. Поскольку 
\begin{subequations}\label{ho}
\begin{gather}
h(o)\geq u(o)=0, 
\tag{\ref{ho}o}\label{hoo}
\\ 
0\overset{\eqref{hoo}}{\leq} \sup_Dh\overset{\eqref{hM}}{=}
 \sup_{x\in D}\int_{\partial D}u {\,{\mathrm d}} \updelta_x^{\complement D}
\leq \sup_{x\in D} \updelta_x^{\complement D}(\partial D)\sup_{\partial D}u\overset{\eqref{upg}}{=} \sup_{\partial D}u,
\tag{\ref{ho}D}\label{hoD}
\end{gather}
\end{subequations}
неравенство \eqref{hmh} влечёт за собой неравенства
\begin{multline}\label{hdis}
h(x)\overset{\eqref{hoo}}{\geq} h(x)-h(o)\overset{\eqref{hmh}}{\geq}
-\bigl( {\sf dist}_{{\sf har}}^D(o,x)-1\bigr)\Bigl(\sup_D h-h(o)\Bigr)
\\
\overset{\eqref{hoo}}{\geq}
-\bigl( {\sf dist}_{{\sf har}}^D(o,x)-1\bigr)\sup_D h
\overset{\eqref{hoD}}{\geq}
-\bigl( {\sf dist}_{{\sf har}}^D(o,x)-1\bigr)\sup_{\partial D}u
\quad\text{для любоой точки $x\in  D$.}
\end{multline} 
По формуле Пуассона\,--\,Йенсена \cite[4.5]{Rans}, \cite[5.7.4]{HK} для субгармонической на $\overline D$ функции $u$ 
\begin{equation}\label{PJ}
u(x)= h(x) -\int
{\sf g}_x^{D}(y){\,{\mathrm d}} \varDelta_u
\quad\text{для всех точек $x\in D$.}
\end{equation}
Применяя оценку сверху \eqref{repr} предложения \ref{prpG} к мере $\mu=\varDelta_u$ с сужением 
$\varDelta$ на $\overline D$, получаем 
\begin{equation*}
\int
{\sf g}_x^{D}(y){\,{\mathrm d}} \varDelta_u\overset{\eqref{repr}}{\leq}
{\sf N}_x^{\varDelta} (\diameter\!D )\in \overline {\mathbb R}^+ \quad\text{при всех $x\in D$}, 
\end{equation*}
что через  формулу Пуассона\,--\,Йенсена \eqref{PJ} в сочетании с  \eqref{hdis} даёт оценку снизу 
\eqref{infuI20}.
\end{proof}

\section{Оценка функции Грина снизу через расстояние Гарнака}
Следующий результат может быть полезен и в других вопросах, поэтому устанавливается несколько больше, чем необходимо в настоящей статье.
\begin{proposition}\label{gsn}
В условиях \eqref{oBDG+} и обозначениях \eqref{rd}
\begin{equation}\label{gDG}
\inf_{x\in \partial D}{\mathsf g}_o^G(x)
\geq \frac{\Bbbk(R+\text{\dj})-\Bbbk(R)}
{\sup\limits_{x\in \partial D}{\sf dist}_{{\sf har}}^{G\!\setminus\!o}\bigl(x,\partial B_o(R)\bigr)}.
\end{equation}
\end{proposition}
\begin{proof}
Функция ${\mathsf g}_o^G$ --- положительная гармоническая на области $G\!\setminus\!o$ и для любой точки $x\in \partial \overline D$ по определению \ref{dH}   расстояния Гарнака 
\begin{equation}\label{g1}
{\mathsf g}_o^G(x)\overset{\eqref{hard}}{\geq} \frac{1}{{\sf dist}_{{\sf har}}^{G\!\setminus\!o}(x,y)}{\mathsf g}_o^G(y)
\quad \text{для любой точки $y\in G\!\setminus\!o$.}
\end{equation}
Рассмотрим произвольные  точки $y\in \partial B_o(R)$. Очевидно, согласно выбору $R$ в \eqref{rdr} 
и $\text{\dj}$ в \eqref{rdd} шар $B_o(R+\text{\dj})$ содержится в  $G$. 
Отсюда по   принципу подчинения \eqref{prpodG} для функций Грина
\begin{equation}\label{g2}
 {\mathsf g}_o^G(y) \overset{\eqref{prpodG}}{\geq}  {\mathsf g}_o^{B_o(R+\text{\it \dj})}(y)\overset{\eqref{Byr}}{=}\Bbbk(R+\text{\dj})-\Bbbk(R)
\quad\text{при всех $y\in \partial B_o(R)$}
\end{equation}
Таким образом, для всех точек $x\in \partial D$
\begin{equation}\label{g3}
{\mathsf g}_o^G(x)\overset{\eqref{g1}}{\geq}
\frac{1}{{\sf dist}_{{\sf har}}^{G\!\setminus\!o}(x,y)}\bigl(\Bbbk(R+\text{\dj})-\Bbbk(R)\bigr)
\quad\text{при всех $y\in \partial B_o(R)$},
\end{equation}
где левая часть не зависит от $y$, а переход  к точной верхней грани по 
всем $y\in \partial B_o(R)$ даёт
\begin{multline*}
{\mathsf g}_o^G(x)\overset{\eqref{g3}}{\geq}
\sup_{y\in \partial B_o(R)}
\frac{1}{{\sf dist}_{{\sf har}}^{G\!\setminus\!o}(x,y)}\bigl(\Bbbk(R+\text{\dj})-\Bbbk(R)\bigr)
=\frac{1}{\inf\limits_{y\in \partial B_o(R)} {\sf dist}_{{\sf har}}^{G\!\setminus\!o}(x,y)}\bigl(\Bbbk(R+\text{\dj})-\Bbbk(R)\bigr)
\\
=\frac{1}{{\sf dist}_{{\sf har}}^{G\!\setminus\!o}\bigl(x,\partial B_o(R)\bigr)}\bigl(\Bbbk(R+\text{\dj})-\Bbbk(R)\bigr)
\quad\text{при всех $x\in \partial D$}.
\end{multline*}
Применяя точную нижнюю грань по $x\in \partial D$ к крайним частям этих (не)равенств, получаем 
\begin{equation*}
\inf_{x\in \partial D}{\mathsf g}_o^G(x)
\geq \inf_{x\in \partial D}\frac{1}{{\sf dist}_{{\sf har}}^{G\!\setminus\!o}\bigl(x,\partial B_o(R)\bigr)}\bigl(\Bbbk(R+\text{\dj})-\Bbbk(R)\bigr)=
\frac{\Bbbk(R+\text{\dj})-\Bbbk(R)}
{\sup\limits_{x\in \partial D}{\sf dist}_{{\sf har}}^{G\!\setminus\!o}\bigl(x,\partial B_o(R)\bigr)},
\end{equation*}
что даёт \eqref{gDG} и доказывает предложение \ref{gsn}.
\end{proof}

\section{Оценка меры Рисса через расстояние Гарнака}\label{Sec5}

\begin{proposition}\label{cord}
Пусть в условиях \eqref{oBDG+} и обозначениях \eqref{rdr}--\eqref{rdd}  функция  $u\not\equiv -\infty$  субгармоническая   на  $\overline G$. Тогда для её  меры Рисса $\varDelta_u$ имеет место неравенство 
\begin{equation}\label{de}
\varDelta_u(\overline D)\leq 
\frac{\sup\limits_{x\in \partial D}{\sf dist}_{{\sf har}}^{G\!\setminus\!o}\bigl(x,\partial B_o(R)\bigr)}{\Bbbk(R+\text{\dj})-\Bbbk(R)}
\Bigl( {\mathsf H}_u^G(o)-u(o)\Bigr),
\end{equation}
где  ${\mathsf H}_u^G$ --- наилучшая гармонической мажоранта функции $u$ внутри  $G$, определённая в  \eqref{hM}.
\end{proposition}

\begin{proof} При $u(o)=-\infty$ неравенство \eqref{de} очевидно, поскольку в его правой части  в этом случае $+\infty$.  В  соотношениях   
$$
\varDelta_u(\overline D)= \frac{1}{\inf\limits_{\overline D}{\mathsf g}_o^G}\int_{\overline D} 
\inf\limits_{\overline D}{\mathsf g}_o^G {\,{\mathrm d}} \varDelta_u(y) \leq \frac{1}{\inf\limits_{\overline D}{\mathsf g}_o^G}\int_{\overline D} 
{\mathsf g}_o^G (y) {\,{\mathrm d}} \varDelta_u(y)\leq 
\frac{1}{\inf\limits_{\partial D}{\mathsf g}_o^G}\int_{\overline G} 
{\mathsf g}_o^G (y) {\,{\mathrm d}} \varDelta_u(y)
$$
в завершающем неравенстве использован принцип минимума для  супергармонической, согласно \eqref{iig},  в $G\supset \overline D$ функции Грина ${\mathsf g}_o^G$. По формуле Пуассона\,--\,Йенсена 
последний  интеграл равен разности $ {\mathsf H}_u^G(o)-u(o)$ при $u(o)\neq -\infty$, откуда, используя предложение \ref{gsn}, получаем 
$$
\varDelta_u(\overline D)\leq \frac{1}{\inf\limits_{\partial  D}{\mathsf g}_o^G}\Bigl( {\mathsf H}_u^G(o)-u(o)\Bigr)
\overset{\eqref{gDG}}{\leq}
\frac{\sup\limits_{x\in \partial D}{\sf dist}_{{\sf har}}^{G\!\setminus\!o}\bigl(x,\partial B_o(R)\bigr)}{\Bbbk(R+\text{\dj})-\Bbbk(R)}
\Bigl( {\mathsf H}_u^G(o)-u(o)\Bigr),
$$
что доказывает неравенство \eqref{de}. 
\end{proof}
Теперь мы можем дать 
\begin{proof}[Доказательство неравенства \eqref{rd} теоремы \ref{th1_1}]
Для сла\-г\-а\-е\-м\-о\-го  ${\sf N}_x^{\varDelta}(\diameter\!D)$ из заключительной части уже доказанного неравенства \eqref{infuI20}
по предложению \ref{prkN} с мерой $\mu :=\varDelta$ и радиальными считающими функциями 
 $\varDelta_x^{\text{\tiny \rm rad}}$ с центром $x$ меры $\varDelta$ имеем
\begin{align*}
{\sf N}_x^{\varDelta} (\diameter\!D)& \overset{\eqref{reprk}}{=}
\int_0^{\diameter\!D}\bigl( \Bbbk (\diameter\!D)-\Bbbk (t)\bigr){\,{\mathrm d}} \varDelta_x^{\text{\tiny \rm rad}}(t) 
=\int_0^{r_x}\bigl( \Bbbk (r_x)-\Bbbk (t)\bigr){\,{\mathrm d}} \varDelta_x^{\text{\tiny \rm rad}}(t)
\\&+
\int_0^{r_x}\bigl( \Bbbk (\diameter\!D)-\Bbbk (r_x)\bigr){\,{\mathrm d}} \varDelta_x^{\text{\tiny \rm rad}}(t)
+\int_{r_x}^{\diameter\!D}\bigl( \Bbbk (\diameter\!D)-\Bbbk (t)\bigr){\,{\mathrm d}} \varDelta_x^{\text{\tiny \rm rad}}(t)
\\&\overset{\eqref{reprk}}{=}{\sf N}_x^{\varDelta} (r_x)+
\bigl( \Bbbk (\diameter\!D)-\Bbbk (r_x)\bigr)\varDelta_x^{\text{\tiny \rm rad}}(r_x)
+\int_{r_x}^{\diameter\!D}\bigl( \Bbbk (\diameter\!D)-\Bbbk (t)\bigr){\,{\mathrm d}} \varDelta_x^{\text{\tiny \rm rad}}(t)
\end{align*}
откуда ввиду возрастания функции \eqref{kKd-2} можем продолжить неравенства как 
\begin{align*}
{\sf N}_x^{\varDelta}(\diameter\!D)&\overset{\eqref{kKd-2}}{\leq}
{\sf N}_x^{\varDelta} (r_x)+
\bigl( \Bbbk (\diameter\!D)-\Bbbk (r_x)\bigr)\varDelta_x^{\text{\tiny \rm rad}}(r_x)
+\int_{r_x}^{\diameter\!D}\bigl( \Bbbk (\diameter\!D)-\Bbbk (r_x)\bigr){\,{\mathrm d}} \varDelta_x^{\text{\tiny \rm rad}}(t)\\
&={\sf N}_x^{\varDelta} (r_x)+
\bigl( \Bbbk (\diameter\!D)-\Bbbk (r_x)\bigr)\varDelta_x^{\text{\tiny \rm rad}}(r_x)
+\bigl( \Bbbk (\diameter\!D)-\Bbbk (r_x)\bigr)\bigl(\varDelta_x^{\text{\tiny \rm rad}}({\diameter\!D})
-\varDelta_x^{\text{\tiny \rm rad}}(r_x)\bigl)\\
&={\sf N}_x^{\varDelta} (r_x)+
\bigl( \Bbbk (\diameter\!D)-\Bbbk (r_x)\bigr)\varDelta_x^{\text{\tiny \rm rad}}({\diameter\!D})
\quad\text{в каждой точке  $x\in D$.}
\end{align*}
Отсюда, применяя к $\varDelta_x^{\text{\tiny \rm rad}}({\diameter\!D})=\varDelta(\overline D)$ предложение 
\ref{cord} при $u(o)=0$, получаем 
\begin{align*}
{\sf N}_x^{\varDelta}(\diameter\!D)&\overset{\eqref{de}}{\leq}
\bigl(\Bbbk (\diameter\!D)-\Bbbk (r_x)\bigr)
\frac{\sup\limits_{x\in \partial D}{\sf dist}_{{\sf har}}^{G\!\setminus\!o}\bigl(x,\partial B_o(R)\bigr)}{\Bbbk(R+\text{\dj})-\Bbbk(R)}
\Bigl( {\mathsf H}_u^G(o)-u(o)\Bigr)
+{\sf N}_x^{\varDelta} (r_x)\\
&\overset{\eqref{hoD}}{\leq}
\bigl(\Bbbk (\diameter\!D)-\Bbbk (r_x)\bigr)
\frac{\sup\limits_{x\in \partial D}{\sf dist}_{{\sf har}}^{G\!\setminus\!o}\bigl(x,\partial B_o(R)\bigr)}{\Bbbk(R+\text{\dj})-\Bbbk(R)}
\sup_{\partial G}u
+{\sf N}_x^{\varDelta} (r_x),
\end{align*}
что по доказанному  неравенству \eqref{infuI20} 
 даёт \eqref{rd}. 
\end{proof}

\begin{proof}[Доказательство следствия \ref{cor1}]
Поскольку множество $(-\infty)_u$ из   \eqref{keyiu} для субгармонической функции $u\not\equiv -\infty$ 
на $\overline G\supset D\supset S$ имеет нулевую ёмкость \eqref{CapE} и, как следствие,
нулевую  $h$-меру Хаусдорфа \eqref{hH} для любой функции $h$ со свойством $h(0)=0$ при радиусе $r$ \cite{Carleson}, \cite[5.2.1]{HK}, то в доказательстве можно пренебречь рассмотрением $S\cap (-\infty)_u$.

Рассмотрим сначала точки $x\in S$, для которых выполнены неравенства 
\begin{equation}\label{vDxt}
\varDelta_x(t)\leq h(t)\sup_{\partial G}u \quad\text{при всех  $t\in [0, r]$}.
\end{equation}
Для таких точек $x$, удовлетворяющих \eqref{vDxt}, получаем
$$
{\sf N}_{x}^{\varDelta}(r)\overset{\eqref{muyrN}}{=}
\widehat{\tt d}\int_0^r\frac{\varDelta_x^{\text{\tiny \rm rad}}(t)}{t^{{\tt d}-1}} {\,{\mathrm d}} t
\overset{\eqref{vDxt}}{\leq}  \widehat{\tt d}\int_0^r\frac{h(t)\sup_{\partial G}u}{t^{{\tt d}-1}} {\,{\mathrm d}} t
\overset{\eqref{N0h}}{=}N_0^h(r)\sup\limits_{\partial G}u\overset{\eqref{N0h}}{<}+\infty.
$$
Это вместе с оценкой \eqref{infuI2} теоремы \ref{th1_1} при выборе $r_x:=r$ для точек $x\in S$, удовлетворяющих \eqref{vDxt}, даёт неравенства  
\begin{multline}\label{kor}
u(x)\geq -\bigl({\sf dist}_{{\sf har}}^{D}(o,x) -1\bigr)\sup_{\partial D} u- 
\frac{\Bbbk(\diameter\!D)-\Bbbk(r)}{\Bbbk(R+\text{\dj})-\Bbbk(R)}\sup\limits_{y\in \partial D}{\sf dist}_{{\sf har}}^{G\!\setminus\!o}\bigl(y,\partial B_o(R)\bigr)\sup_{\partial G} u
 -N_0^h(r)\sup\limits_{\partial G}u 
\\
\overset{\eqref{oBDG+},\eqref{hard}}{\geq} -\bigl({\sf dist}_{{\sf har}}^{D}(o,x) -1\bigr) \sup_{\partial G} u- 
\frac{\Bbbk(\diameter\!D)-\Bbbk(r)}{\Bbbk(R+\text{\dj})-\Bbbk(R)}\sup\limits_{y\in \partial D}{\sf dist}_{{\sf har}}^{G\!\setminus\!o}\bigl(y,\partial B_o(R)\bigr)\sup_{\partial G} u
 -N_0^h(r)\sup\limits_{\partial G}u 
\\
\geq -\biggl(
\sup_{y\in S}{\sf dist}_{{\sf har}}^{D}(o,y)  -1
+\frac{\Bbbk(\diameter\!D)-\Bbbk(r)}{\Bbbk(R+\text{\dj})-\Bbbk(R)}\sup\limits_{y\in \partial D}{\sf dist}_{{\sf har}}^{G\!\setminus\!o}\bigl(y,\partial B_o(R)\bigr)+N_0^h(r)\biggr)
\sup_{\partial G} u,
\end{multline}
т.е. вне исключительного множества $E$ всех точек $x\in S$, для которых нарушено свойство \eqref{vDxt}.
Применение точной нижней грани по всем точкам $x\in S\setminus E$ к крайним частям соотношений  \eqref{kor}, где правая часть уже не зависит от $x\in S\setminus E$, приводит к требуемой оценке  \eqref{{eEr}u}. 

Для оценки размеров исключительного множества 
\begin{equation}\label{Eyx}
E:=\Bigl\{ x\in S\Bigm|  \varDelta_x(t_x)\geq  h(t_x)\sup_{\partial G}u
\quad\text{для некоторого $t_x\in (0,r]$} \Bigr\}
\end{equation}
будет использована 
\begin{theoB}[{\cite[лемма  3.2]{Landkof}, \cite[2.8.14]{Federer},   \cite{FL}, \cite[теорема 2]{GriKri10}, \cite[I.1, Замечания]{Gusman}, \cite{Sullivan}, \cite{Krantz}}] Пусть  $t\colon x\underset{x\in E}{\longmapsto} {\mathbb R}^+\setminus 0$ --- функция на  $E\subset {\mathbb R}^{\tt d}$. Если  $t$ или $E$ ограничены, то  найдётся  последовательность попарно  различных точек $x_j\in E$, $j\in N\subset {\mathbb N}$,  для которых $E\subset \bigcup\limits_{j\in N} \overline B_{x_j}\bigl(t(x_j)\bigr)$ и пересечение любых $5^{\tt d}$  различных шаров $\overline B_{x_j}\bigl(t(x_j)\bigr)$ пусто.
\end{theoB}
Применяя теорему Безиковича к множеству $E$ с функцией $t(x)\overset{\eqref{Eyx}}{\underset{x\in E}{:=}}t_x$
и используя предложение \ref{cord}, для некоторого не более чем счётного числа шаров $\overline B_{x_j}\bigl(t(x_j)\bigr)$, 
$j\in N\subset {\mathbb N}$, покрывающего $E$ с кратностью не более чем $5^{\tt d}$, получаем 
\begin{equation*}
\sum_j h(t_{x_j})
\overset{\eqref{Eyx}}{\leq} \sum_{j} \frac{\varDelta_{x_j}(t_{x_j})}{\sup\limits_{\partial G}u}
\leq 5^{\tt d}\varDelta(\overline D) \frac{1}{\sup\limits_{\partial G}u}
\overset{\eqref{de}}{\leq} 
 \frac{5^{\tt d}}{\sup\limits_{\partial G}u}
\frac{\sup\limits_{x\in \partial D}{\sf dist}_{{\sf har}}^{G\!\setminus\!o}\bigl(x,\partial B_o(R)\bigr)}{\Bbbk(R+\text{\dj})-\Bbbk(R)}
\Bigl( {\mathsf H}_u^G(o)-u(o)\Bigr),
\end{equation*}
где 
$$
\Bigl( {\mathsf H}_u^G(o)-u(o)\Bigr)\leq \sup\limits_{\partial G}u-u(o)\overset{\eqref{oBDG+}}{=}\sup\limits_{\partial G}u,
$$
что даёт 
\begin{equation*}
\sum_j h(t_{x_j})\leq 
{5^{\tt d}}
\frac{\sup\limits_{x\in \partial D}{\sf dist}_{{\sf har}}^{G\!\setminus\!o}\bigl(x,\partial B_o(R)\bigr)}{\Bbbk(R+\text{\dj})-\Bbbk(R)}, 
\quad t_{x_j}\overset{\eqref{Eyx}}{\underset{j\in N}{\in}}(0, r].
\end{equation*}
Последнее по определению \ref{defH} $h$-обхвата ${\mathfrak m}_h^{\text{\tiny $r$}}$ в \eqref{mr}
и означает, что выполнено \eqref{{eEr}E}.
\end{proof}
 
\begin{remark}
Применения оценок расстояния  Гарнака, установленных ранее в   \cite[\S~2, \S\S~9--11, предложение 11.3]{KhaKhaChe08I}, 
\cite[гл.~III]{KhaKhaChe08II}, \cite{Kha21H}, к полученным в настоящей статье оценкам снизу субгармонических функций  и иным задачам предполагается изложить  в последующих статьях. 
\end{remark}

\renewcommand\refname{\bf Литература}

\bibliographystyle{amsplain}

\end{document}